\documentclass[11pt,oneside,letterpaper]{article}
\usepackage{amsmath,bm}
\usepackage{amsfonts}
\usepackage{amssymb}
\usepackage{amsthm}
\usepackage{graphicx}
\usepackage{rotating}
\usepackage{multirow}
\usepackage{xcolor}
\usepackage{rotfloat}
\usepackage{pgfplots}
\usepackage[T1]{fontenc}
\usepackage{amssymb,amsmath}
\usepackage[left=1in, bottom=1in, right=1in, top=1in]{geometry}
\usepackage{times}
\usepackage{bbm, dsfont}
\usepackage{natbib}
\usepackage{cases}
\usepackage{url}
\usepackage{etoolbox}
\apptocmd{\UrlBreaks}{\do\-}{}{}
\bibpunct{(}{)}{;}{a}{,}{,}
\usepackage{float}
\usepackage{pdfsync}
\usepackage{setspace}
\usepackage{enumerate}
\usepackage{accents}
\usepackage[utf8]{inputenc}
\usepackage[T1]{fontenc}
\usepackage{authblk}
\newcommand{\ubar}[1]{\underaccent{\bar}{#1}}
\numberwithin{equation}{section}
\usepackage[shortlabels]{enumitem}
\theoremstyle{definition}\newtheorem{definition}{Definition}
\newtheorem{lemma}{Lemma}

\newtheorem{theorem}{Theorem}

\theoremstyle{definition}
\theoremstyle{definition}
\theoremstyle{definition}
\theoremstyle{definition}

\title{The Impossibility of Testing for Dependence Using Kendall's $\tau$ Under Missing Data of Unknown Form}
\author[1]{Oliver R. Cutbill}
\author[1]{Rami V. Tabri\thanks{Corresponding author: rami.tabri@sydney.edu.au.}}
\affil{The University of Sydney}
\begin{document}
\maketitle
\begin{abstract}
\noindent This paper discusses the statistical inference problem associated with testing for dependence between two continuous random variables using Kendall's $\tau$ in the context of the missing data problem. We prove the worst-case identified set for this measure of association always includes zero. The consequence of this result is that robust inference for dependence using Kendall's $\tau$, where robustness is with respect to the form of the missingness-generating process, is impossible.  \\ \\
\noindent AMS 2020 subject classifications: 62H15; 62D10; 62G10\\
Keywords: Impossible Inference; Statistical Dependence; Kendall's $\tau$; Partial Identification; Missing Data.
\end{abstract}
\section{Introduction}
Testing for statistical dependence between two random variables is an important facet of theoretical and empirical statistical research, and arises as a problem of interest in various areas of the natural and social sciences. Applications in social science include the study of the relationship between health outcomes and insurance levels (e.g.,~\citealp{cameron1993}), survey analysis (e.g.,~\citealp{opinionsurveys}), stress-testing risk-management models (e.g.,~\citealp{asimit2016}), and stock market co-movements~(e.g.,~\citealp{horvath2015,cameron1993}). In the natural sciences, applications arise in contexts as diverse as cancerous somatic alteration co-occurrences (e.g.,~\citealp{discover}) and the movement of animals across time (e.g.,~\citealp{animals}).

\par Tests for dependence based on Kendall's $\tau$ (\citealp{kendall1938}) constitute a relevant tool in empirical practice to detect monotonic dependence between two random variables. The interested reader may refer, for instance, to the monographs~\cite{nelsen2006},~\cite{BKN-2011}, and the references therein. The strength of such testing procedures is that $\tau$ is a distribution-free measure of association between paired continuous random variables. In particular, let $(X,Y)$ be a pair of continuous random variables having joint distribution $H$ and marginal distributions $F$ and $G$, respectively. In moment form, this measure of association for the random vector $(X,Y)$ is defined as
\begin{align}\label{eq - Kendall tau}
\tau=4 E_H\left[C\left(F(X),G(Y)\right)\right]-1,
\end{align}
where $C$ is the copula of $H$, and $E_H$ denotes the expectation operator with respect to the distribution $H$. The hypothesis testing problem for detecting monotonic dependence using $\tau$ in~(\ref{eq - Kendall tau}) has the form
\begin{align}\label{eq - Kendall tau hypotheses}
H_0: \tau=0\quad\text{versus}\quad H_1:\tau\neq 0.
\end{align}
The null hypothesis in~(\ref{eq - Kendall tau hypotheses}) posits no monotonic dependence between the two random variables, and the alternative hypothesis is the negation of the null.

\par Statistical procedures for the hypothesis test problem~(\ref{eq - Kendall tau hypotheses}) are predicated on the assumption that the random vector $(X,Y)$ is observable. However, this assumption is violated in empirical practice because datasets can have missing values. For example, missing data can arise in the form of nonresponse, as in self-reported cross-sectional and longitudinal surveys, which is inevitable, or at follow-up in clinical studies. See, for example,~\cite{Dutz-2021} for a discussion on the prevalence of nonresponse in economics research. Missing data are also universal in ecological and evolutionary data, as in other branches of science; see, for example, the monograph~\cite{Ecological-book} and the references therein. Imputation methods are commonly used to address the missing data problem and enable testing with a complete dataset. However, the validity of such tests hinges on the correct specification of the imputation procedure, which can lead to biased inferences if misspecified. Another approach in the literature that addresses the missing data problem imposes assumptions on the missingness-generating process (MGP) that point-identify $\tau$ in~(\ref{eq - Kendall tau}). In the context of Kendall's rank-correlation test, see, for example,~\cite{yecanadapaper} who assume the MGP is either missing completely at random or weakly exogenous, and~\cite{Ma} who assumes that it is either missing at random or missing completely at random. While practical, such tests also ignore misspecification of the MGP, which weakens the credibility of any derived inferences (\citealp{manski2003}).

\par Consequently, we ask if it is possible to conduct non-parametric inference for dependence using $\tau$ under missing data of unknown form. The results of this paper imply that such robust inference is impossible. Reasoning from first principles, any sensible testing procedure of this sort must be based on $\tau$'s identified set because it characterizes the information about this parameter contained in the observables. The identified set for this parameter is an interval $[\ubar{\tau},\bar{\tau}]\subseteq [-1,1]$ whose bounds depend on observables. Therefore, the testing problem for inferring statistical dependence using this information must have the form
\begin{align}\label{eq - test problem}
H_{0}: \ubar{\tau}\leq0\;\text{and}\;\bar{\tau}\geq 0\quad \text{versus}\quad H_1: \ubar{\tau}>0\;\text{or}\;\bar{\tau}<0.
\end{align}
Under $H_1$ in~(\ref{eq - test problem}), the identified set is a subset of either $[-1,0)$ or $(0,1]$. Since $\ubar{\tau}\leq\tau\leq \bar{\tau}$ holds by definition, this hypothesis implies that either $\tau<0$ or $\tau>0$ holds, so that $X$ and $Y$ are statistically dependent. We show the bounds satisfy the inequalities $\ubar{\tau}\leq 0\leq \bar{\tau}$, \emph{for all} joint distributions of $(X,Y)$ and MGPs. These inequalities show the null hypothesis in~(\ref{eq - test problem}) always holds, implying that one cannot partition the underlying probability model into two submodels that are compatible with the assertions of the null and alternative hypotheses. Therefore, the worst-case bounds are useless in detecting dependence between $X$ and $Y$ through the testing problem~(\ref{eq - test problem}). We prove that this property of the bounds holds in the setup where the marginal distributions are known to the practitioner, which implies that it holds in the setup where those distributions are unknown. The reason is that the bounds in the case where the marginal distributions are unknown must be less informative, and hence, weakly wider than their counterparts under known marginal distributions, implying that they must also satisfy this property. A critical step in our theoretical derivations is an innovative use of results on extremal dependence described in~\cite{puccetti2}.

\par This paper contributes to the literature on impossible inference, which has a rich history that started with the classic paper of~\cite{Bahadur-Savage}. The recent paper by~\cite{BERTANHA2020247} connects this literature and presents a taxonomy of the types of impossible inferences. Our result falls under Type A in their taxonomy, as the alternative is indistinguishable from the null. However, it should be noted that our result is not a consequence of the model of the null hypotheses being dense in the set of all likely models with respect to the total-variation distance, which is the essential characteristic of Type A impossible inferences. Rather, it flows from the fact that the bounds $\ubar{\tau}$ and $\bar{\tau}$ are uninformative because they do not define a partition of the underlying probability model into two submodels that are compatible with the assertions of $H_0$ and $H_1$ in~(\ref{eq - test problem}).

\par The idea of using bounds to account for missing data problems started with the seminal paper of~\cite{Manski-89} and gained popularity with the important paper of~\cite{manskihorowitz1995}. Since then, there has been a growing influential literature on partial identification that has shaped empirical practice; see, for example,~\cite{canay2016} for a recent survey of this literature and the references therein. Inference on bounds to account for missing data in moment inequality models have been considered in a variety of settings, such as distributional analyses (e.g.,~\citealp{Blundell-Gosling-Ichimura-Meghir}), treatment effect (e.g.,~\citealp{Lee-2009}), and stochastic dominance testing (e.g.,~\citealp{FAKIH2021}). In contrast to those works, this paper shows that such an approach is futile in testing for dependence under missing data of unknown form using Kendall's $\tau$ and its worst-case bounds. We also discuss how to obtain informative partitions of the underlying probability model through restrictions on the dependence between $X$ and $Y$ and/or the MGPs.

\par There is also a strand in the partial identification literature focusing on
parameters that depend on the joint distribution of two random variables with point-identified marginal distributions; see, for example,~\cite{Fan-Patton} for a survey of this strand and the references therein. However, to the best of our knowledge, this literature strand has not considered partial identification of those parameters arising from the missing data problem. While convenient, the point-identification of the margins can be untenable in applications with missing data and can create challenges in the inference for such parameters --- the results of our paper exemplify this point.

\par The rest of this paper is organized as follows. Section~\ref{Section - Setup and Preliminaries} introduces the statistical setup of this paper and preliminary results on extremal dependence that we utilize in the proof of our results. Section~\ref{Section - Results} presents our results, Section~\ref{Section - Discussion} discusses the scope of the results and implications for empirical practice. Section~\ref{Section - Conclusion} concludes. All proofs are relegated to the Appendix.

\section{Setup and Preliminaries}\label{Section - Setup and Preliminaries}
Consider the random vector $(X,Y,Z)$ having joint distribution $P$, where $X$ and $Y$ are the random variables of interest, which are continuous, and $Z$ is a categorical variable supported on $\{1,2,3,4\}$ indicating missingness on $X$ and $Y$. In this setup,
\begin{align}\label{eq - practitioner Observes}
    \text{ the practitioner observes } \begin{cases}
    (X,Y), & \text{ if } Z = 1 \\
    (X,\ast), & \text{ if } Z = 2 \\
    (\ast,Y), & \text{ if } Z = 3 \\
    (\ast,\ast), & \text{ if } Z = 4
    \end{cases}
\end{align}
where $\ast$ denotes the missing variable. For simplicity, we assume that the marginal cumulative distribution functions (CDFs) of $X$ and $Y$, denoted by $F$ and $G$, respectively, are known by the practitioner. We derive the worst-case bounds on $\tau$ under the following probability model.
\begin{definition}\label{Definition - prob model}
Let $\mathcal{P}_{F,G}$ be the set of distributions of the random vector $(X,Y,Z)$ supported on $\mathcal{X}\times\mathcal{Y}\times\{1,2,3,4\}\subseteq\mathbb{R}^2\times\{1,2,3,4\}$, with generic element $P$, such that
\begin{enumerate}[(i)]
\item $P$ has a density, $p$.
\item $(X,Y)$ is a continuous random vector having strictly positive density.
\item $X$ has CDF $F$ and $Y$ has CDF $G$.
\end{enumerate}
\end{definition}
\noindent The worst-case bounds on $\tau$ without the practitioner's knowledge of the marginal CDFs can be computed by extremizing their counterparts with known $F$ and $G$ over feasible candidate values of these CDFs. We elaborate on this point in Section~\ref{Section - Results}, and show impossible inference for dependence under $\mathcal{P}_{F,G}$ implies that it holds in the more general scenario where $F$ and $G$ are unknown to the practitioner.

\par In this setup, a MGP is specified through restrictions on the joint distribution of $(X,Y,Z)$. The model $\mathcal{P}_{F,G}$ does not place any restrictions on the dependence between $(X,Y)$ and $Z$ beyond the existence of a density. To account for the missing data problem, we exhibit $\tau$ as a functional of $P\in\mathcal{P}_{F,G}$. For each $P\in\mathcal{P}_{F,G}$, an application of the Law of Total probability shows the corresponding value of $\tau$ has the following representation:
\begin{align}\label{eq - Kendall tau MDP}
\tau(P)=4\sum_{z=1}^{4}E_{P}\left[C_P\left(F(X),G(Y)\right)\mid Z=z\right]\,P[Z=z]-1,
\end{align}
where $C_P$ is the copula of the joint CDF $P(X\leq \cdot,Y\leq \cdot)$.\footnote{The existence and uniqueness of $C_P$ in our setup is a result due to~\cite{Sklar}.} This representation of $\tau$ is useful since it clarifies the situation faced by the practitioner in our setup. In particular, it shows that $\tau$ can be calculated for each $P\in\mathcal{P}_{F,G}$ using the following parameters: the copula $C_P$; the conditional CDFs, $P(X\leq \cdot,Y\leq \cdot \mid Z = z)$ for $z=1,2,3,4$; and the marginal probabilities of $Z$, $\left\{P(Z = z)\right\}_{z=1}^{4}$. From sampling, asymptotically the practitioner can recover $\left\{P(Z = z)\right\}_{z=1}^{4}$ and $P(X\leq \cdot,Y\leq \cdot \mid Z = 1)$ but not $C_P$ and $\left\{P(X\leq \cdot,Y\leq \cdot \mid Z = z)\right\}_{z=2}^{4}$, as the data alone do not contain any information on the latter. Consequently, a MGP can be characterized in terms of a specification of the conditional CDFs $\left\{P(X\leq \cdot,Y\leq \cdot \mid Z = z)\right\}_{z=1}^{4}$.

\par The above analysis shows $\tau$ is partially identified in the missing data setting when were are agnostic about the MGP. The identified set of $\tau$ in this case is a closed interval subset of $[-1,1]$ whose boundary corresponds to the worst-case bounds on $\tau$. These bounds permit the entire spectrum of MGPs, which is especially useful when the data have a large number of missing values, as there can be a diversity of explanations for it.

\par The next section describes the worst-case bounds on $\tau$ in~(\ref{eq - Kendall tau MDP}) and raises the statistical issues concerning testing for dependence between $X$ and $Y$ in a manner that is robust to the MGP. In developing our results we make use of the Fr\'{e}chet-Hoeffding copula bounds and two results on extreme values of means of supermodular functions from~\cite{puccetti2}. To describe these results, denote by $\mathcal{C}$ the set of all bivariate copulas on the unit square $[0,1]^2$. The Fr\'{e}chet-Hoeffding bounds are $\max\{u+v-1,0\}\leq C(u,v)\leq \min\{u,v\}$ for all $(u,v)\in[0,1]^2$, which hold for all $C\in\mathcal{C}$.

\par A function $s:\mathbb{R}^2 \to \mathbb{R}$ is called \emph{supermodular} if for all $x_1 \leq x_2$ and $y_1 \leq y_2$,
\begin{align*}
    s(x_1,y_1)+s(x_2,y_2) \geq s(x_1,y_2) + s(x_2,y_1),
\end{align*}
 important examples of which are copulas. This point is important as the bounds on $\tau$ are characterized in terms of copulas of $(X,Y)$'s joint distribution. The results of \cite{puccetti2} that we utilize are Theorems~2.1 and~3.1 in their paper and we restate them in the following lemma, but in a form that is more suitable for the derivation of our results.
\begin{lemma}[\cite{puccetti2}]\label{lemma-extremal}
Let $s:\mathbb{R}^2 \to \mathbb{R}$ be a supermodular function, and $X$ and $Y$ be random variables with marginal CDFs $F$ and $G$ respectively. Furthermore, let $\mathcal{C}$ be as described above.
\begin{enumerate}
\item The moment $E_H \left( s\left(X,Y\right) \right)$, when viewed as a functional of the copula $C$ through the representation $H=C(F,G)$, is maximized when $X$ and $Y$ are co-monotonic. That is,
\begin{align}
    \label{lemma-eqn-extremal-1}
    \sup \{E_H \left( s\left(X,Y\right) \right): H=C(F,G),C\in\mathcal{C} \} = E_{H^*} \left( s\left(X,Y\right) \right),
\end{align}
where $H^{*}=\min\left\{F,G\right\}$.
\item  The functional $E_H \left( s\left(X,Y\right) \right)$, when viewed as a functional of the copula $C$ through the representation $H=C(F,G)$,  is minimized when $X$ and $Y$ are counter-monotonic. That is,
\begin{align}
    \label{lemma-eqn-extremal-2}
    \inf \{E_H \left( s\left(X,Y\right) \right): H=C(F,G),C\in\mathcal{C}  \} = E_{H^\dagger} \left( s\left(X,Y\right) \right),
\end{align}
where $H^\dagger=\max\left\{F+G-1,0\right\}$.
\end{enumerate}
\end{lemma}

\section{Results}\label{Section - Results}
The first result characterizes the worst-case bounds on Kendall's $\tau$ in the case where the marginal CDFs of $(X,Y)$ are known.
\begin{theorem}\label{Thm - WC bounds}
Let $\mathcal{P}_{F,G}$ be given as in Definition~\ref{Definition - prob model}, and suppose that $P\in\mathcal{P}_{F,G}$. The worst-case bounds of $\tau$ under the distribution $P$ that satisfy $\ubar{\tau}(P) \leq \tau \leq \bar{\tau}(P)$, are given by:
\begin{align*}
\bar{\tau}(P) & = 4 E_{P}\left[\min\left\{F(X),G(Y)\right\} \mid Z = 1\right]P(Z = 1) + 4 E_{P}\left[F(X) \mid Z = 2\right]P(Z = 2) \\
& \qquad +4E_{P}\left[G(Y) \mid Z = 3\right]P(Z = 3) + 4P(Z = 4) - 1,\quad \text{and}\\
\ubar{\tau}(P) & = 4 E_{P}\left[\max\left\{F(X)+G(Y)-1,0\right\} \mid Z = 1\right]P(Z = 1) - 1.
\end{align*}
\end{theorem}
\begin{proof}
See Appendix~\ref{Appendix - Thm1 proof}.
\end{proof}
\noindent The result of this theorem is that for each $P\in\mathcal{P}_{F,G}$ we can determine bounds on $\tau$ that permit the entire spectrum of MGPs. For each $P\in\mathcal{P}_{F,G}$, these bounds are sharp; that is, any value in the interval $[\ubar{\tau}(P),\bar{\tau}(P)]$, including the endpoints, cannot be rejected as the true value of $\tau(P)$. This property of the bounds follows from the sharpness of the Fr\'{e}chet-Hoeffding bounds on a bivariate copula, which we use in the derivation of $\ubar{\tau}(P)$ and $\bar{\tau}(P)$.

\par To test for statistical dependence using $\tau$ in a manner that is robust to the form of the MGP, one can only consider tests that depend on observables through $\tau$'s identified set. This means positing the hypothesis testing problem
\begin{align}\label{eq - testing problem}
H_{0}:P_0\in\mathcal{P}^{0}_{F,G}\quad\text{versus}\quad H_{1}:P_0\in\mathcal{P}^{1}_{F,G},
\end{align}
where $P_0$ is the true distribution of $(X,Y,Z)$, $\mathcal{P}^{0}_{F,G}=\{P\in\mathcal{P}_{F,G}:\bar{\tau}(P)\geq 0\;\text{and}\;\ubar{\tau}(P)\leq0\}$ and $\mathcal{P}^{1}_{F,G}=\{P\in\mathcal{P}_{F,G}:\bar{\tau}(P)< 0\;\text{or}\;\ubar{\tau}(P)>0\}$. Notice that $\mathcal{P}^{1}_{F,G}$ is the relative complement of $\mathcal{P}^{0}_{F,G}$ in $\mathcal{P}_{F,G}$; that is, $\mathcal{P}^{1}_{F,G}=\mathcal{P}_{F,G}-\mathcal{P}^{0}_{F,G}$, which implies that either $\tau(P_0)<0$ or $\tau(P_0)>0$ holds. The next result implies that $\mathcal{P}^{1}_{F,G}=\emptyset$, meaning the null hypothesis in~(\ref{eq - testing problem}) is always true.
\begin{theorem}\label{Thm - Impossibility}
Let $\mathcal{P}_{F,G}$ be given as in Definition~\ref{Definition - prob model}. Furthermore, let $\ubar{\tau}(P)$ and $\bar{\tau}(P)$ be given as in Theorem~\ref{Thm - WC bounds}. Then $\ubar{\tau}(P) \leq 0\leq \bar{\tau}(P)$ for every $P\in\mathcal{P}_{F,G}$.
\end{theorem}
\begin{proof}
See Appendix~\ref{Appendix - Thm2 proof}.
\end{proof}

\par The result of Theorem~\ref{Thm - Impossibility} shows the worst-case bounds of $\tau$ are not informative in the sense that they do not simultaneously take negative or positive values when the joint distribution of $(X,Y)$ exhibits negative or positive dependence, respectively. This property creates an impossibility in testing for dependence on the basis of $\tau$ that is robust to missingness of any form, as in~(\ref{eq - testing problem}), since it implies that $\mathcal{P}^{1}_{F,G}=\emptyset$.

\par The results of Theorem~\ref{Thm - WC bounds} and~\ref{Thm - Impossibility} have assumed that the marginal CDFs of $X$ and $Y$ are known to the practitioner. In the scenario where this is not the case, neither of those distributions would be point-identified when we are agnostic about the nature of the MGP. By an application of the Law of Total Probability to $F$ and $G$, we can obtain pointwise bounds on these marginal CDFs as $\underline{F}(x) \leq F(x) \leq \overline{F}(x)$ and $\underline{G}(y) \leq G(y) \leq \overline{G}(y)$ for all $(x,y)\in\mathcal{X}\times\mathcal{Y}$, with the boundaries themselves CDFs, given by
\begin{align*}
\overline{F}(x) &=  P(X\leq x\mid Z=1) +P(X\leq x\mid Z=2)P[Z=2]+ P[Z=3] + P[Z=4]\;\forall x\in\mathcal{X}, \\
\underline{F}(x) &= \begin{cases} P(X\leq x\mid Z=1) +P(X\leq x\mid Z=2)P[Z=2] &\text{ if } x \in \mathcal{X} \\
    1 &\text{ if } x = \sup \mathcal{X},
    \end{cases}
\end{align*}
and
\begin{align*}
\overline{G}(y) &=  P(Y\leq y\mid Z=1) +P(Y\leq y\mid Z=3)P[Z=3]+ P[Z=2] + P[Z=4]\;\forall y\in\mathcal{Y}, \\
\underline{G}(x) &= \begin{cases} P(Y\leq y\mid Z=1) +P(Y\leq y\mid Z=3)P[Z=3] &\text{ if } y \in \mathcal{Y} \\
    1 &\text{ if } y = \sup \mathcal{Y}.  \\
    \end{cases}
\end{align*}
Denoting by $\mathcal{F}$ and $\mathcal{G}$ the sets of all CDFs of $X$ and $Y$ that satisfy the respective bounds described above, the probability model is $\mathcal{P}=\bigcup_{F\in\mathcal{F},G\in\mathcal{G}}\mathcal{P}_{F,G}$. Thus, one has bounds on $\tau(P)$ that depend on hypothetical values of the margins $F$ and $G$, and extremizing these bounds with respect to
the margins over $\mathcal{F}$ and $\mathcal{G}$ yields the worst-case upper and lower bounds, $\sup_{(F,G)\in\mathcal{F}\times\mathcal{G}}\bar{\tau}(P)$ and $\inf_{(F,G)\in\mathcal{F}\times\mathcal{G}}\ubar{\tau}(P)$, respectively. Therefore, the conclusion of Theorem~\ref{Thm - Impossibility} also holds for these worst-case bounds since they are wider than their counterparts in the scenario where the marginal CDFs of $X$ and $Y$ are known to the practitioner.\footnote{Scrutinizing the expressions of $\bar{\tau}(P)$ and $\ubar{\tau}(P)$, observe that the worst-case bounds in this larger model can be obtained in closed-form. For the upper bound, replace $F$ and $G$ in $\bar{\tau}(P)$ with $\overline{F}$ and $\overline{G}$, respectively; and for the lower bound, replace $F$ and $G$ in $\ubar{\tau}(P)$ with $\underline{F}$ and $\underline{G}$, respectively.}

\section{Discussion}\label{Section - Discussion}
\par This section discusses the implications of our results. The model $\mathcal{P}_{F,G}$ is large, which is the reason why the bounds $\bar{\tau}$ and $\ubar{\tau}$ do not yield a partition of $\mathcal{P}_{F,G}$ that is compatible with the hypotheses in~(\ref{eq - testing problem}). Note that the model is non-parametric and permits (i) the entire spectrum of MGPs, and (ii) all bivariate absolutely continuous copulas in modelling the statistical dependence between $X$ and $Y$. This point raises the following question: does restricting $\mathcal{P}_{F,G}$ give rise to an identified set for $\tau$ whose bounds are informative in detecting dependence between $X$ and $Y$? The answer is in the affirmative. Restrictions on $\mathcal{P}_{F,G}$ can be motivated by many considerations based on the application at hand. For example, they can be arise from the possession of side-information or by restrictions implied by economic theory as in the partial identification approach in econometrics (e.g.,~\citealp{tamer2010}). We elaborate on this point with an example of the former utilizing results in~\cite{Nelson-et-al-2001}.

\par Let $P_0\in\mathcal{P}_{F,G}$ denote the true distribution of $(X,Y,Z)$. Suppose we possess side-information that $P_0\left(X\leq \tilde{x},Y\leq \tilde{y}\right)=\theta,$ where $\tilde{x}$ and $\tilde{y}$ are the medians of $X$ and $Y$, and $\theta\in[0,1/2]$. Accordingly, we must have $C_{P_0}(1/2,1/2)=\theta$, which represents the side-information in terms of the copula. Theorem~1 of~\cite{Nelson-et-al-2001} provides the bounds on the copula under this restriction, which are given by
\begin{align*}
\ubar{C}_{\theta}(u,v) & =\max\left\{\max\{u+v-1,0\},\theta-(1/2-u)^{+}-(1/2-v)^{+}\right\}\quad\text{and}\\
\bar{C}_{\theta}(u,v) & =\min\left\{\min\{u,v\},\theta+(u-1/2)^{+}+(v-1/2)^{+}\right\},
\end{align*}
for all $(u,v)\in[0,1]^{2}$, where $a^{+}=\max\{a,0\}$. Thus, the bounds on the joint distribution of $(X,Y)$ are
\begin{align}
\ubar{C}_{\theta}(F(x),G(y))\leq P\left(X\leq x,Y\leq y\right)\leq \bar{C}_{\theta}(F(x),G(y)),\quad\forall (x,y)\in\mathcal{X}\times\mathcal{Y},
\end{align}
which hold for all $P\in\mathcal{P}_{F,G,\theta}$ where the probability model $\mathcal{P}_{F,G,\theta}=\left\{P\in\mathcal{P}_{F,G}: P\left(X\leq \tilde{x},Y\leq \tilde{y}\right)=\theta\right\}$ accounts for the side-information. We can apply identical steps as in the proof of Theorem~\ref{Thm - WC bounds} to obtain the corresponding bounds on $\tau$ under the model $\mathcal{P}_{F,G,\theta}$, but replacing the Fr\'{e}chet-Hoeffding lower and upper bounds with $\ubar{C}_{\theta}$ and $\bar{C}_{\theta}$, respectively. For brevity, we omit these details. The bounds on $\tau(P)$ are given by
\begin{align*}
\bar{\tau}_{\theta}(P) & = 4 E_{P}\left[\bar{C}_{\theta}(F(X),G(Y)) \mid Z = 1\right]P(Z = 1) + 4 E_{P}\left[F(X) \mid Z = 2\right]P(Z = 2) \\
& \qquad +4E_{P}\left[G(Y) \mid Z = 3\right]P(Z = 3) + 4P(Z = 4) - 1\quad \text{and}\\
\ubar{\tau}_{\theta}(P) & = 4 E_{P}\left[\ubar{C}_{\theta}(F(X),G(Y))\mid Z = 1\right]P(Z = 1) - 1,
\end{align*}
and satisfy $\left[\ubar{\tau}_{\theta}(P),\bar{\tau}_{\theta}(P)\right]\subset\left[\ubar{\tau}(P),\bar{\tau}(P)\right]$, for $P\in\mathcal{P}_{F,G,\theta}$.

\par In contrast to the worst-case bounds on $\tau$, the bounds $\bar{\tau}_{\theta}$ and $\ubar{\tau}_{\theta}$ are informative, in the sense that $\exists P_1,P_2,P_3\in \mathcal{P}_{F,G,\theta}$ such that
\begin{align*}
\bar{\tau}_{\theta}(P_1),\tau(P_1)<0,\;\ubar{\tau}_{\theta}(P_2),\tau(P_2)>0,\;\text{and}\;\ubar{\tau}_{\theta}(P_3)\leq \tau(P_3)\leq \bar{\tau}_{\theta}(P_3)\;\text{with}\;\tau(P_3)=0.
\end{align*}
We demonstrate this point using a numerical example in which we specify $X$ and $Y$ being uniformly distributed on $[0,1]$ for simplicity. Furthermore, we set $\theta=0.4$ and derive the MGP from a multinomial logit specification for the propensity probabilities; i.e.,
\begin{align}\label{eq - mulitnomial logit}
P\left[Z=z\mid X=x,Y=y\right]=\frac{e^{\gamma_{1,z}x+\gamma_{2,z}y}}{\sum_{j=1}^{4}e^{\gamma_{1,j}x+\gamma_{2,j}y}}\quad\forall z=1,2,3,4.
\end{align}
Finally, to complete the specification of $P$ we must designiate the copula of $(X,Y)$, $C_P$, as the marginal probabilities of $Z$ can be obtained by integrating the propensity scores with respect to this copula. Then, by Bayes' Theorem, the MGP is given by the conditional probability density functions
\begin{align*}
P\left[X=x,Y=y\mid Z=z\right]=P\left[Z=z\mid X=x,Y=y\right]\left[\frac{C_P(x,y)}{P[Z=z]}\right]\quad\forall z=1,2,3,4.
\end{align*}
We set $C_P$ as the bivariate Gaussian copula, with standard normal margins, and construct $P_1$, $P_2$, and $P_3$ through setting the correlation coefficient $\rho$. As $\bar{\tau}_{\theta}(P_1),\ubar{\tau}_{\theta}(P_2),\bar{\tau}_{\theta}(P_3)$, and $\ubar{\tau}_{\theta}(P_3)$ are linear combinations of moments, we calculated them using Monte Carlo simulations with $10^8$ random draws from the corresponding bivariate Gaussian copula.

\par The parameter specification for $P_1$ are as follows: $\gamma_{1,1}=\gamma_{2,1}=2$; $\gamma_{1,2}=-5,\gamma_{2,2}=1/4$;$\gamma_{1,3}=5,\gamma_{2,3}=-1/4$;$\gamma_{1,4}=\gamma_{2,4}=-5$; and $\rho=-0.999$. This yields $\tau(P_1)<0$ and $\bar{\tau}_\theta(P_1)\approx-0.0108$. The parameter specification for $P_2$ are as follows: $\gamma_{1,1}=\gamma_{2,1}=1/2$; $\gamma_{1,2}=3,\gamma_{2,2}=1/2$;$\gamma_{1,3}=1/2,\gamma_{2,3}=-2$;$\gamma_{1,4}=\gamma_{2,4}=2$; and $\rho=0.99$. This specification has $\tau(P_2)>0$ and gives rise to $\ubar{\tau}_\theta(P_1)\approx0.034$. Finally,  the parameter specification for $P_3$ are identical to that under $P_2$ except that now $\rho=0$. This specification has $\tau(P_3)=0$ and gives rise to $\ubar{\tau}_\theta(P_3)\approx-0.32$ and  $\bar{\tau}_\theta(P_3)\approx0.63$.

\par The numerical results demonstrate the refined bounds can be informative in the detection of dependence. In such a situation, the practitioner can consider the following testing problem
\begin{align}\label{eq - testing problem Refined}
H_{0}:P_0\in\mathcal{P}_{F,G,\theta}^{0}\quad\text{versus}\quad H_{1}:P_0\in\mathcal{P}_{F,G,\theta}^{1},
\end{align}
where $\mathcal{P}_{F,G,\theta}^{0}=\{P\in\mathcal{P}_{F,G,\theta}:\bar{\tau}_{\theta}(P)\geq 0\,\text{and}\,\ubar{\tau}_{\theta}(P)\leq0\}$ and $\mathcal{P}_{F,G,\theta}^{1}=\{P\in\mathcal{P}_{F,G,\theta}:\bar{\tau}_{\theta}(P)< 0\,\text{or}\,\ubar{\tau}_{\theta}(P)>0\}$ form a partition of $\mathcal{P}_{F,G,\theta}$. The bounds are linear combinations of moments but with unknown coefficients being the marginal probabilities $P[Z=z]$. As these marginal probabilities are typically estimable at the $\sqrt{n}$-rate, one can adapt moment inequality testing procedures for this situation, which are abundant and well-established (e.g.,~\citealp{Andrews-Soares,Canay,RSW}). Developing the details of such a testing procedure goes beyond the intended scope of the paper, and is left for future research.

\par As this side-information only restricts the dependence between $X$ and $Y$, any valid testing procedure that rejects $H_0$ in favour of $H_1$ in~(\ref{eq - testing problem Refined}) would be robust to the nature of the MGP. This robustness, however, comes at the expense of an ambiguity under the null. Specifically, $\exists P\in\mathcal{P}_{F,G,\theta}$ such that $\tau(P)\neq 0$ and $\ubar{\tau}_{\theta}(P)\leq 0\leq \bar{\tau}_{\theta}(P)$. The uninformative nature of $H_0$ in~(\ref{eq - testing problem Refined}) is a consequence of circumventing assumptions on the MGP, which are unverifiable in practice. If one fails to reject the null, then, unfortunately, one cannot conclude anything informative about the dependence between $X$ and $Y$. In such a situation, we recommend empirical researchers perform a sensitivity analysis of this empirical conclusion (i.e., non-rejection of $H_1$) with respect to plausible assumptions on the MGP. The virtue of this type of analysis is that it establishes, in a transparent way, clear links between empirical outcomes and different assumptions made on the MGP. Such an analysis would reveal non-trivial links between assumptions on the MGP and inferences made. See, for example, ~\cite{Blundell-Gosling-Ichimura-Meghir} and~\cite{Lee-2009} who refine worst-case distributional bounds using economic theory and develop testable implications based on them in the contexts of distributional analyses and treatment effect, respectively. See also~\cite{FAKIH2021} who discuss the refinement of worst-case distributional bounds of ordinal variables in the context of stochastic dominance testing by positing assumptions on the form of nonresponse in self-reported surveys.

\section{Conclusion}\label{Section - Conclusion}
\par This paper establishes the impossibility of performing inference for dependence between two continuous random variables using Kendall's $\tau$ under missing data of unknown form. The crux of the issue is that its identified set always includes zero, implying that the sign of $\tau$ is not identified. We show how refining this identified set using additional information can address this problem, creating a pathway for robust inference based on statistical procedures from the moment inequality testing literature.

\section{Acknowledgement}
We thank Brendan K. Beare and Christopher D. Walker for helpful feedback and comments.

\bibliographystyle{chicago}
\bibliography{thesisbib}

\appendix
\section{Proofs of Results}
\subsection{Proof of Theorem~\ref{Thm - WC bounds}}\label{Appendix - Thm1 proof}
\begin{proof}
The proof proceeds by the direct method. We shall derive bounds on $\tau(P)$ for each $P\in\mathcal{P}_{F,G}$ and recall that
\begin{align*}
\tau(P)=4\sum_{z=1}^{4}E_{P}\left[C_P\left(F(X),G(Y)\right)\mid Z=z\right]\,P[Z=z]-1.
\end{align*}
First, we focus on the upper bound, and bound each term appearing in the sum separately. Starting with $E_{P}\left[C_P\left(F(X),G(Y)\right)\mid Z=1\right]$, note that it is less than or equal to
$E_{P}\left[\min\left\{F(X),G(Y)\right\}\mid Z=1\right]$, since by Fr\'{e}chet-Hoeffding upper bound in $2$-dimensions we have that $C_P\left(F(x),G(y)\right)\leq \min\left\{F(x),G(y)\right\}$ for all $(x,y)\in\mathcal{X}\times\mathcal{Y}$. Now focusing on the term $E_{P}\left[C_P\left(F(X),G(Y)\right)\mid Z=2\right]$, note that $Y$ is not observed and we replace $G(Y)$ it with its largest theoretical value, $1$. Thus we bound $C_P(F(x),G(y)) \leq C_P(F(x),1)= F(x)$ which holds for all $(x,y)\in\mathcal{X}\times\mathcal{Y}$. Therefore, $E_{P}\left[C_P\left(F(X),G(Y)\right)\mid Z=2\right]\leq E_{P}\left[F(X)\mid Z=2\right]$. Similarly, we bound $C_P(F(x),G(y)) \leq C_P(1,G(y))=G(y)$ which holds for all $(x,y)\in\mathcal{X}\times\mathcal{Y}$. Hence, $E_{P}\left[C_P\left(F(X),G(Y)\right)\mid Z=3\right]\leq E_{P}\left[G(Y)\mid Z=3\right]$. Finally, on the event $\{Z=4\}$, which is when neither $X$ nor $Y$ is observed, we bound $C_P\left(F(x),G(y)\right)$ from above by $1$, which holds for all $(x,y)\in\mathcal{X}\times\mathcal{Y}$. This yields $E_{P}\left[C_P\left(F(X),G(Y)\right)\mid Z=4\right]\leq 1$. Combining these bounds yields
\begin{align*}
\tau(P) & \leq 4 E_{P}\left[\min\left\{F(X),G(Y)\right\} \mid Z = 1\right]P(Z = 1) + 4 E_{P}\left[F(X) \mid Z = 2\right]P(Z = 2) \\
& \qquad +4E_{P}\left[G(Y) \mid Z = 3\right]P(Z = 3) + 4P(Z = 4) - 1 \\
 & = \bar{\tau}(P).
\end{align*}

\par Next, we focus on the lower bound, and bound each term appearing in the representation of $\tau(P)$ separately. Starting with $E_{P}\left[C_P\left(F(X),G(Y)\right)\mid Z=1\right]$, note that it is greater than or equal to \\
$E_{P}\left[\max\left\{F(X)+G(Y)-1,0\right\}\mid Z=1\right]$, since by Fr\'{e}chet-Hoeffding lower bound in $2$-dimensions we have that $C_P\left(F(x),G(y)\right)\geq \max\left\{F(x)+G(y)-1,0\right\}$ for all $(x,y)\in\mathcal{X}\times\mathcal{Y}$. Now focusing on the term \\ $E_{P}\left[C_P\left(F(X),G(Y)\right)\mid Z=2\right]$, note that $Y$ is not observed and we replace $G(Y)$ it with its smallest theoretical value, $0$. Thus, we bound $C_P(F(x),G(y)) \geq \max\left\{F(x)+G(y)-1,0\right\} \geq \max\left\{F(x)-1,0\right\}=0$, which holds for all $(x,y)\in\mathcal{X}\times\mathcal{Y}$, implying that $E_{P}\left[C_P\left(F(X),G(Y)\right)\mid Z=2\right]\geq 0$. Similarly we bound $C_P(F(x),G(y)) \geq \max\left\{F(x)+G(y)-1,0\right\} \geq 0$, which holds for all $(x,y)\in\mathcal{X}\times\mathcal{Y}$, so that $E_{P}\left[C_P\left(F(X),G(Y)\right)\mid Z=3\right]\geq 0$. Finally, on the event $\{Z=4\}$, which is when neither $X$ nor $Y$ is observed, we bound $E_{P}\left[C_P\left(F(X),G(Y)\right)\mid Z=4\right]$ from below by $0$. Combining these bounds yields
\begin{align*}
\tau(P) \geq 4 E_{P}\left[\max\left\{F(X)+G(Y)-1,0\right\} \mid Z = 1\right]P(Z = 1) - 1 =\ubar{\tau}(P).
\end{align*}
This concludes the proof.
\end{proof}
\subsection{Proof of Theorem~\ref{Thm - Impossibility}}\label{Appendix - Thm2 proof}
\begin{proof}
The proof proceeds by the direct method. We split the proof into two cases: (i) showing $\ubar{\tau}(P) \leq 0$ for every $P\in\mathcal{P}_{F,G}$, and (ii) showing $\bar{\tau}(P) \geq 0$ for every $P\in\mathcal{P}_{F,G}$.

\par {\bf Part (i)}. Fix $P\in\mathcal{P}_{F,G}$. Firstly, note that if $P(Z=1)=0$, then $\ubar{\tau}(P)=-1$, and the inequality trivially holds. Now, we consider the case $P(Z=1)>0$, and note that in this case
\begin{align}\label{proof - thm2 - 0}
\ubar{\tau}(P)=4P(Z=1)\int_{\mathcal{X}\times\mathcal{Y}}\max\left\{F(x)+G(y)-1,0\right\}\,p(x,y\mid Z=1)\,dx dy-1,
\end{align}
where we have used the Condition (i) in the definition of $\mathcal{P}_{F,G}$ to express the integrator in terms of a density. By Condition (ii), we can re-write $p(x,y\mid Z=1)$ as follows:
\begin{align*}
p(x,y\mid Z=1)=\frac{p(x,y,1)}{P(Z=1)}=\frac{p(x,y,1)}{p_{X,Y}(x,y)}\,\frac{p_{X,Y}(x,y)}{P(Z=1)}=P(Z=1\mid X=x, Y=y)\,\frac{p_{X,Y}(x,y)}{P(Z=1)}.
\end{align*}
Substituting this expression for $p(x,y\mid Z=1)$ in~(\ref{proof - thm2 - 0}) and simplifying yields
\begin{align}
\ubar{\tau}(P) & =4\int_{\mathcal{X}\times\mathcal{Y}}\max\left\{F(x)+G(y)-1,0\right\}\,P(Z=1\mid X=x, Y=y)\,p_{X,Y}(x,y)\,dx dy -1\nonumber\\
 & \leq 4\int_{\mathcal{X}\times\mathcal{Y}}\max\left\{F(x)+G(y)-1,0\right\}\,p_{X,Y}(x,y)\,dx dy-1,\label{proof - thm2 - 1}
\end{align}
where the inequality follow from $P(Z=1\mid X=x, Y=y)\leq 1$ holding for all $(x,y)\in\mathcal{X}\times\mathcal{Y}$. We can maximize the integral in~(\ref{proof - thm2 - 1}) with respect to the joint distribution of $(X,Y)$ to find that
\begin{align}\label{proof - thm2 - 2}
\ubar{\tau}(P)\leq 4\sup_{P\in\mathcal{P}_{F,G}}E_{P}\left[\max\left\{F(X)+G(Y)-1,0\right\}\right]-1.
\end{align}
Since the function $(x,y)\mapsto\max\left\{F(x)+G(y)-1,0\right\}$ is supermodular, Part 1 of Lemma~\ref{lemma-extremal} implies
\begin{align*}
\sup_{P\in\mathcal{P}_{F,G}}E_{P}\left[\max\left\{F(X)+G(Y)-1,0\right\}\right]\leq E_{H^*} \left[\max\left\{F(X)+G(Y)-1,0\right\}\right],
\end{align*}
where $H^{*}=\min\left\{F,G\right\}$. Now, we shall argue that $E_{H^*} \left[\max\left\{F(X)+G(Y)-1,0\right\}\right]=1/4$. As the CDFs $F$ and $G$ are known, we define the new random variables $U := F(X)$ and $V := G(X)$. By the Probability Integral Transform, both $U$ and $V$ are distributed as continuous uniform on the unit interval $[0,1]$. This yields the representation
\begin{align*}
E_{H^*} \left[\max\left\{F(X)+G(Y)-1,0\right\}\right]=E_{C^*} \left[\max\left\{U+V-1,0\right\}\right].
\end{align*}
where $C^*(u,v)=\min(u,v)$. This copula is supported on the line segment $u=v$ in the unit square $[0,1]^2$. Consequently,
$$E_{C^*} \left[\max\left\{U+V-1,0\right\}\right]=\int_{1/2}^{1}(2u-1)\,du =1/4.$$ Therefore, $\ubar{\tau}(P)\leq 0$. Since we have chosen an arbitrary $P\in\mathcal{P}_{F,G}$, the deduction $\ubar{\tau}(P)\leq 0$ holds for all $P\in\mathcal{P}_{F,G}$. This concludes the proof for the lower-bound.

\par {\bf Part (ii)}. First, recall that
\begin{align*}
\bar{\tau}(P) & = 4 E_{P}\left[\min\left\{F(X),G(Y)\right\} \mid Z = 1\right]P(Z = 1) + 4 E_{P}\left[F(X) \mid Z = 2\right]P(Z = 2) \\
& \qquad +4E_{P}\left[G(Y) \mid Z = 3\right]P(Z = 3) + 4P(Z = 4) - 1.
\end{align*}
In the case that $P(Z=4)=1$ holds, we have that $\bar{\tau}(P)=3\geq 0$, which is the desired result. Now we consider the case where $P(Z=4)\in[0,1)$. Substituting $P(Z = 4)=1-\sum_{z=1}^{3}P(Z = z)$ into the expression of $\bar{\tau}(P)$ above and simplifying yields
 \begin{align}
\bar{\tau}(P) & = 4 E_{P}\left[\min\left\{F(X),G(Y)\right\}-1 \mid Z = 1\right]P(Z = 1) + 4 E_{P}\left[F(X)-1 \mid Z = 2\right]P(Z = 2)\nonumber  \\
& \qquad +4E_{P}\left[G(Y)-1 \mid Z = 3\right]P(Z = 3) + 3\nonumber \\
& \geq 4\sum_{z=1}^{3}E_{P}\left[\min\left\{F(X),G(Y)\right\}-1 \mid Z = z\right]P(Z = z)+3,\label{proof - thm2 - 3}
\end{align}
where the inequality~(\ref{proof - thm2 - 3}) arises from the fact that $F(X)-1$ and $G(Y)-1$ are bounded from below by $\min\left\{F(X),G(Y)\right\}-1$ with probability one (under $P$). Next, fix an arbitrary $z\in\{1,2,3\}$ such that $P(Z=z)>0$. Note that such a $z$ exists since $P(Z=4)\in[0,1)$. We will re-express
$$E_{P}\left[\min\left\{F(X),G(Y)\right\}-1 \mid Z = z\right]P(Z = z)$$ in a more convenient form to apply Part 2 of Lemma~\ref{lemma-extremal}. It is by definition equal to
\begin{align*}
P(Z=z)\int_{\mathcal{X}\times\mathcal{Y}}\left(\min\left\{F(x),G(y)\right\}-1\right)\frac{p(x,y,z)}{P(Z=z)}\,dxdy.
\end{align*}
Now since $P\in\mathcal{P}_{F,G}$ we can multiply and divide by $p_{X,Y}(x,y)$ in the integrand and simplifying yields
\begin{align*}
\int_{\mathcal{X}\times\mathcal{Y}}\left(\min\left\{F(x),G(y)\right\}-1\right)\frac{p(x,y,z)}{p_{X,Y}(x,y)}p_{X,Y}(x,y)\,dxdy,
\end{align*}
which is equal to
\begin{align*}
\int_{\mathcal{X}\times\mathcal{Y}}\left(\min\left\{F(x),G(y)\right\}-1\right)P\left[Z=z\mid X=x,Y=y\right]p_{X,Y}(x,y)\,dxdy.
\end{align*}
Note that this re-writing applies for each $z=1,2,3$ for which $P(Z=z)>0$. Now because
\begin{align*}
0\leq P\left[Z=z\mid X=x,Y=y\right] & \leq 1 \quad \forall (x,y,z)\in\mathcal{X}\times\mathcal{Y}\times \{1,2,3,4\}\,\text{and}\\
\min\left\{F(x),G(y)\right\}-1&\leq 0\quad \forall (x,y)\in\mathcal{X}\times\mathcal{Y},
\end{align*}
 the expression in~(\ref{proof - thm2 - 3}) is bounded from below by
\begin{align}\label{proof - thm2 - 5}
4\int_{\mathcal{X}\times\mathcal{Y}}\left(\min\left\{F(x),G(y)\right\}-1\right)p_{X,Y}(x,y)\,dxdy+3=4E_{P}\left[\min\left\{F(X),G(Y)\right\}-1\right]+3,
\end{align}
which in turn is bounded from below by
\begin{align}\label{proof - thm2 - 6}
4\inf_{P\in\mathcal{P}_{F,G}}E_{P}\left[\min\left\{F(X),G(Y)\right\}-1\right]+3.
\end{align}
As the function $(x,y)\mapsto\min\left\{F(x),G(y)\right\}-1$ is supermodular, and the integrator in~(\ref{proof - thm2 - 6}) only depends on the joint distribution of $(X,Y)$, we can apply Part 2 of Lemma~\ref{lemma-extremal} to deduce that the minimal value~(\ref{proof - thm2 - 6}) is bounded from below:
\begin{align}\label{proof - thm2 - 7}
\inf_{P\in\mathcal{P}_{F,G}}E_{P}\left[\min\left\{F(X),G(Y)\right\}-1\right]\geq E_{H^\dagger} \left[\min\left\{F(X),G(Y)\right\}-1\right],
\end{align}
where $H^{\dagger}=\max\left\{F+G-1,0\right\}$. Now we argue that the right side of the inequality in~(\ref{proof - thm2 - 7}) equals -3/4 to deduce the result of this theorem. This expected value equals
$$\left(\int_0^{1/2} (u - 1) \,du - \int_{1/2}^1 u \,du\right)=-3/4.$$
Now, using this result, we find that
\begin{align}\label{proof - thm2 - 8}
\bar{\tau}(P)\geq 4\inf_{P\in\mathcal{P}_{F,G}}E_{P}\left[\min\left\{F(X),G(Y)\right\}-1\right]+3\geq0,
\end{align}
concluding the proof.
\end{proof}
\end{document}